\documentclass[12pt]{amsart}

\title[Strong wavefront lemma] {Strong wavefront lemma and counting
  lattice points in sectors} 
\dedicatory{Preliminary version as of \today} 
\author{Alexander Gorodnik, Hee Oh and Nimish Shah}
\address{Mathematics 253-37\\ Caltech\\Pasadena, CA 91106}
\email{gorodnik@caltech.edu} \address{Math Department\\151 Thayer
  St.\\Brown University\\Providence, RI 02912}
\email{heeoh@math.brown.edu} \address{School of Mathematics, TIFR\\ 1
  Homi Bhabha Road\\ Mumbai, 400005, India}
\email{nimish@math.tifr.res.in}

\thanks{The first and the second authors partially supported by NSF
  0400631 and NSF 0333397 respectively}

\usepackage{epsfig}
\usepackage{graphicx}
\usepackage{amscd,amssymb}
\theoremstyle{plain}
\newtheorem{thm}[equation]{Theorem}

\newtheorem{cor}[equation]{Corollary}
\newtheorem{Prop}[equation]{Proposition}
\newtheorem{Lem}[equation]{Lemma}
\newtheorem{lem}[equation]{Lemma}
\newtheorem{rem}[equation]{Remark}

\numberwithin{equation}{section}

\newcommand{\mc}[1]{{}}
\newcommand{\q}{\mathbb{Q}}

\newcommand{\e}{\varepsilon}
\newcommand{\z}{\mathbb{Z}}
\newcommand{\Z}{\z}
\renewcommand{\q}{\mathbb{Q}}
\newcommand{\n}{\mathbb{N}}

\newcommand{\br}{\mathbb{R}}
\newcommand{\R}{\br}
\newcommand{\N}{\n}
\newcommand{\inv}{^{-1}}

\newcommand{\cS}{\mathcal S}

\newcommand{\Cal}{\mathcal}

\newcommand{\abs}[1]{\left|#1\right|}
\newcommand{\norm}[1]{\left\|#1\right\|}

\newcommand{\cO}{{\mathcal{O}}}
\newcommand{\cQ}{{\mathcal{Q}}}

\newcommand{\supp}{\hbox{\rm supp}}

\newcommand{\mbq}{\text{\boldmath $q$}}

\newcommand{\vol}{\operatorname{Vol}}

\newcommand{\SL}{\operatorname{SL}}
\newcommand{\SO}{\operatorname{SO}}
\newcommand{\la}{\mathfrak}

\newcommand{\diag}{\operatorname{diag}}

\newcommand{\Ad}{\operatorname{Ad}}

\newcommand{\ad}{\operatorname{ad}}

\newcommand{\ignore} [1] {}

\newcommand{\cW}{{\mathcal{W}}}

\begin{document}

\begin{abstract}
  We compute the asymptotics of the number of integral quadratic forms
  with prescribed orthogonal decompositions and, more generally,
  the asymptotics of the number of lattice points lying in sectors of affine
  symmetric spaces. A new key ingredient in this article is the strong
  wavefront lemma, which shows that the generalized Cartan
  decomposition associated to a symmetric space is uniformly Lipschitz.
\end{abstract}

\maketitle

\section{Introduction}\label{sec:intro}

One of the motivations of this paper is a certain counting problem in
the space of quadratic forms. Let $\cS_W$ be the vector space of all
quadratic forms on a Euclidean space $W$ of dimension $d$. We fix an
integral structure on $W$, and hence on $\cS_W$. Let $\mathcal{Q}_W$
denote the subset of $\cS_W$ consisting of quadratic forms of
determinant $\pm 1$, and set $\cQ_W(\Z)=\cQ_W\cap \cS_W(\Z)$. Let
$\norm{\cdot}$ be any norm on $\cS_W$. It follows from the main result
of Duke, Rudnick and Sarnak \cite{drs}, as well as, Eskin and McMullen
\cite{em} that for $d\geq 3$ there exists a constant $c>0$ such that
\begin{equation} \label{eq:N_T-ball}
\#\{\mbq\in \mathcal{Q}_W(\mathbb{Z}):\, \|\mbq\|<T\}\sim_{T\to\infty}
c\cdot T^{d(d-1)/2}.
\end{equation}

Here we will consider a refinement of this
problem that concerns counting quadratic forms with prescribed
structure.  Fix an orthogonal decomposition
\begin{equation}\label{eq:W}
  W=\oplus_{i=0}^n W_i,
\end{equation}
and for $\Omega\subset \hbox{SO}(W)$ and $\Omega'\subset
\mathcal{Q}_{W_0}\times\cdots \times \mathcal{Q}_{W_n}$, set
\begin{equation} \label{eq:N_T}
N_T(\Omega,\Omega')=\#\left\{\mbq\in\mathcal{Q}_W(\mathbb{Z}):\,
  \begin{tabular}{l}
    $\|\mbq\|<T,$\\
    $\mbq(k\cdot x)=a_0 \mbq_0(x)+\dots + a_n \mbq_n(x)$\\
    for some $k\in \Omega$, $(\mbq_0,\dots,\mbq_n)\in \Omega'$,\\
    and $a_0>\dots>a_n>0$ 
  \end{tabular}
\right\}.
\end{equation}

For example, if we choose $W_i$'s to be one dimensional, then we are
counting the number of quadratic forms in a ball of radius $T$ which
can be diagonalized via conjugation by an element from a prescribed
set $\Omega$ of orthogonal transformations to obtain a form with distinct eigenvalues in
decreasing order of absolute values, and with prescribed sign ($\pm$)
in each diagonal entry.

Assuming that $\Omega$ and $\Omega'$ are bounded  measurable
sets such that the subset $\Omega\Omega'$
%$$
%\Omega(\Omega_0+\cdots+\Omega_n)\subset \hbox{SO}(W)
% (\mathcal{Q}_{W_0}\oplus\cdots\oplus \mathcal{Q}_{W_n})
%$$
has positive measure and boundary of measure zero\footnote{The measure
  of $\Omega\Omega'$ is understood in terms of the identification
  \eqref{eq:mea} and \eqref{eq:dm}.}, we prove the following:

\begin{thm}\label{th:quad} For $d\geq 3$,
$$
N_T(\Omega,\Omega')\sim_{T\to\infty} c\cdot T^{d(d-\dim W_n)/2}
$$
for some $c=c(\Omega\Omega')>0$.
\end{thm}

Theorem~\ref{th:quad} is an example of our general result (Theorem
\ref{cor:counting}) on counting
lattice points in sectors of affine symmetric spaces.
In~\cite{drs,em} it is shown that the number of integral points in an
affine symmetric $\q$-variety in a sequence of growing subsets $S_T$
is asymptotic to the volume of $S_T$, provided the sets $S_T$ are {\it
  well-rounded}.  A family of subsets $S_T$ being well-rounded means
roughly that the volumes of neighborhoods of the boundaries of $S_T$
are uniformly negligible compared to the total volumes of $S_T$
(see \eqref{eq:well-rounded} for the precise condition).
In \cite{drs,em}, it is shown that the norm balls are well-rounded.
However, in most situations, given a sequence of subsets $S_T$ which
arises naturally in the geometric or number-theoretic contexts in the
category of affine symmetric spaces, it is highly non-trivial to
determine whether the family $S_T$ is well-rounded.

The main result of this paper is to show that {\em sectors}\/ in
affine symmetric spaces define a well-rounded family of growing
subsets, and consequently, we obtain the asymptotic counting of
lattice points in sectors.  The main technical lemma needed is what we
call `strong wave front lemma', a terminology reflecting it being a
stronger version of the wavefront lemma introduced by Eskin and
McMullen~\cite{em}.

Now we introduce notation that we use throughout the paper.
Let $G$ be a connected noncompact semisimple Lie group with finite center. A closed
subgroup $H$ of $G$ is called symmetric if its identity component
coincides the the identity component of the set of fixed points of an
involution, say $\sigma$, of $G$.  In this case, the homogeneous space
$G/H$ is called an {\em affine symmetric space}. Recall that a maximal
compact subgroup of $G$ is a symmetric subgroup associated to a Cartan
involution on $G$. Affine symmetric spaces have many features similar
to Riemannian symmetric spaces.  In particular, a generalized Cartan
decomposition holds:
\[
G=KAH
\]
where $K$ is a maximal compact subgroup of $G$ compatible with $H$,
and $A$ is a Cartan subgroup corresponding to the pair $(K,H)$.

More precisely, there exists a Cartan involution $\theta$ of $G$ which
commutes with $\sigma$, and let $K=\{g\in G:\theta(g)=g\}$, which is a
maximal compact subgroup of $G$. Let $\la{g}$, $\la{h}$, and $\la{k}$
denote the Lie algebras associated to $G$, $H$ and $K$,
respectively. Let $\theta$ and $\sigma$ also denote their
differentials on $\la{g}$.  Since $H$ and $K$ are $\theta$ stable, we
have the following orthogonal decomposition with respect to the
killing form on $\la{g}$: $\la{g}=\la{k}\oplus \la{p}$, and
$\la{g}=\la{h}\oplus \la{q}$, where $\la{p}$ and $\la{q}$ are the
$(-1)$-eigenspaces of $\theta$ and $\sigma$, respectively. Let
$\la{a}$ denote the maximal abelian subalgebra of $\la{p}\cap\la{q}$
which can be extended to a maximal abelian subalgebra, say $\la{b}$,
of $\la{p}$. Let $A$ denote the analytic subgroup of $G$ associated to
$\la{a}$. This $A$ is called the Cartan subgroup corresponding to the
symmetric pair $(K,H)$.

\bigskip\noindent{\bf Wavefront Lemma (Eskin and McMullen~\cite{em}).}
Given any neighborhood $\cO$ of $e$ in $G$, there exists a
neighborhood $\tilde\cO$ of $e$ in $G$ such that
\[
\tilde\cO g\subset g\cO H,\quad \forall g\in KA.
\]
\bigskip

Next we will strengthen this result for uniformly regular elements of
$g\in G$. For this we will need additional notation
(cf.~\cite[Ch.~7]{sch},\cite[Part~II]{hs} or \cite{gos}). Let
$\la{g}^\alpha$ denote a simultaneous eigenspace for $\ad\la{a}$
action on $\la{g}$ associated to the linear character
$\alpha\in\la{a}^\ast$. Let
$\Sigma_\sigma=\{\alpha\in\la{a}^\ast:\la{g}^\ast\neq0\}$. Then
$\la{g}=\sum_{\alpha\in\Sigma_\sigma\cup\{0\}}\la{g}^\alpha$, and
$\Sigma_\sigma$ forms a root system. Choose a closed positive Weyl
chamber $A^+\subset A$. Let $\Sigma_\sigma^+$ denote the set of
positive roots and $\Delta_\sigma$ the corresponding system of
positive simple roots. The associated Weyl group is given by
$\cW_\sigma=N_K(\la{a})/Z_K(\la{a})$. One can choose a set $\cW\subset
N_K(\la{a})\cap N_K(\la{b})$ of coset representatives of
$N_K(\la{a})/N_{K\cap H}(\la{a})Z_K(\la{a})$. Then 
\begin{equation}\label{eq:cartan0}
G=\cup_{w\in \cW} KA^+{w}H.
\end{equation}

For any $c>0$, an element $g=kawh\in KA^+\cW H$ will be called {\em
  $c$-regular\/} if $\alpha(\log a)\geq c$ for all
$\alpha\in\Delta_\sigma$; (here and later, our notation indicates that
$k\in K$, $a\in A^+$, $w\in\cW$, and $h\in H$).  Otherwise, we call
such an element {\em $c$-singular}.

We fix a Riemannian metric on $G$ and denote by $\mathcal{O}_\e$ the
$\e$-ball at identity.

\begin{thm}[Strong wavefront lemma-I]\label{th:sw}
  Given $c>0$, there exist $\ell>1$ and $\e_0>0$ such that for every
  $c$-regular $g=kawh\in KA^+{w}H$ and $0<\e<\e_0$,
$$
\Cal O_\e\cdot g \subset (K\cap \mathcal{O}_{\ell\e})k\cdot (A\cap
\mathcal{O}_{\ell\e})a\cdot w (H\cap \mathcal{O}_{\ell \e})h.
$$
\end{thm}

% The wavefront lemma, which is a weaker form of Theorem~\ref{th:sw},
% was introduced by Eskin and McMullen \cite[Theorem~3.1]{em} (see also
% \cite[Lemma~5.11]{emm}). It says that for any neighborhood
% $\mathcal{O}$ of $e$ in $G$ there exists a neighborhood
% $\tilde{\mathcal{O}}$ of $e$ in $G$ such that that for every $g\in
% KA$,
% $$
% \tilde{\mathcal{O}}g\subset g\mathcal{O}H.
% $$
The continuity of the Cartan decomposition for Riemannian symmetric
spaces (that is, when $H=K$) was independently shown in Nevo
\cite[Proposition~7.3]{n} and by Gorodnik and Oh \cite[Theorem~2.1]{GO}.
While the proof of \cite{n} uses embeddings of $G$ in linear groups,
the proof of \cite{GO} is based on geometric properties of the
Riemannian symmetric spaces.  The strong wavefront lemma was used in
\cite{n} to prove maximal inequalities for cube averages on semisimple
groups and in \cite{GO} to compute the asymptotics of the number of
lattice points lying in sectors.

Theorem~\ref{th:sw} fails on the set of singular elements; for
example, in $\SL_2(\R)$ if $\Omega$ is a small neighborhood of the
$e$, then $(\Omega\cap K)(\Omega\cap A)(\Omega\cap K)$ does not
contain a neighborhood of the $e$ in $\SL_2(\R)$.  To state a version
of the strong wavefront lemma that holds for singular elements, we
introduce additional notation.  Given $J\subset \Delta_\sigma$, an
element $kawh\in KA^+\cW H$ is called {\em $(J,c)$-regular\/} if
$\alpha(\log a)\geq c$ for all $\alpha\in J$.  Let
$I=\Delta_\sigma\setminus J$. We set $A_I=\exp(\ker I)\subset A$. Let
$M_I$ be the analytic semisimple subgroup whose Lie algebra is
generated by $\la{g}^{\pm\beta}$, $\beta\in \Sigma_\sigma^+\cap
\langle I\rangle$. Then $M_I$ centralizes $A_I$.  Now
\[
G=\cup_{w\in \cW} KM_IA_I^+wH\quad\text{and}\quad M_I\cap A_I=\{e\},
\]
where $A_I^+=A_I\cap A^+$.

\begin{thm}[Strong wavefront lemma-II]\label{th:sw2}
  Given $c>0$, there exist $\ell>1$ and $\e_0>0$ such that for any
  $I\subset\Delta_\sigma$ and $J=\Delta_\sigma\setminus I$, and every
  $g=kawh\in KA^+\cW H$ and $0<\e<\e_0$, if $g$ is $(J,c)$-regular, then
  \[
  \cO_\e\cdot g \subset (K\cap \mathcal{O}_{\ell\e})k\cdot (M_I\cap
  \mathcal{O}_{\ell\e})\cdot (A_I\cap \mathcal{O}_{\ell\e})a\cdot w
  (H\cap \mathcal{O}_{\ell \e})h.
  \]
\end{thm}
  
\begin{rem} \label{label:w-M_I} \rm
    Observe that  by \cite[Corollary~4.7]{gos}, since $wv_0$
    is fixed by the symmetric subgroup $M_I\cap wHw\inv$ of $M_I$, the
    orbit $M_I(wv_0)$ is closed. Since $M_I\subset Z_G(A_I)$,
    we have $M_Iawv_0=aM_Iwv_0$ is closed. Thus, the set $KM_Iawv_0$ is
    closed for any $a\in A_I$. Moreover the natural map $KM_I/(M_I\cap
    wHw\inv)\to KM_Iawv_0$ given by $km(M_I\cap wHw\inv)\mapsto
    kmawv_0$ is a homeomorphism.     
  \end{rem} 
  
A natural generalization of the Cartan decomposition for Riemannian
symmetric spaces is the decomposition
\begin{equation}\label{eq:cartan00}
G=K{\tilde A^+}H
\end{equation}
where $\tilde{A^+}$ is  a Weyl chamber in $A$  with respect to the
Weyl group $(N_G(A)\cap K\cap H)/(Z_G(A)\cap K\cap H)$. In 
Section \ref{sec:last}, we will obtain the strong wavefront lemmas with
respect to the decomposition (\ref{eq:cartan00}), which generalize
Theorem \ref{th:sw} and Theorem \ref{th:sw2}.

\subsection*{Well-roundedness of sectors} 

% We use Theorem~\ref{th:sw2} to solve the lattice counting problem for
% sectors with respect to the decomposition:
% $$
% G=KM_IA_I w H.
% $$

Let $\iota: G\to \hbox{GL}(W)$ be an irreducible representation of $G$
and $v_0\in W$ such that if $H$ denotes the stabilizer of $v_0$ then
$H$ is a symmetric subgroup of $G$. Therefore by
\cite[Corollary~4.7]{gos} the orbit $V=Gv_0$ closed. Hence it can be
realized as an affine symmetric
space $G/H$. Let $\Gamma$ be a lattice in $G$. We
suppose that $H\cap \Gamma$ is also a lattice in $H$. In particular,
$H\Gamma$ is closed in $G$ (, and hence $\Gamma v_0$
is a discrete subset of $W$.  For a norm $\|\cdot\|$ on $W$, we set
$$
B_T=\{w\in W:\, \|w\|<T\}.
$$
It was shown in \cite{drs,em} that the orbit $\Gamma v_0$ is
``equidistributed'' with respect to the sets $V\cap B_T$ in the
following sense:
\begin{equation}\label{eq:equi}
  \# (\Gamma v_0\cap B_T) \sim_{T\to \infty}
  \operatorname{Vol}(V\cap  B_T)
\end{equation}
where $\vol$ is the $G$-invariant measure on $V\cong G/H$ determined
by the Haar measures on $G$ and $H$ chosen such that
$\vol(G/G\cap\Gamma)=\vol(H/H\cap\Gamma)=1$.  In fact, it was shown in
\cite{em} that (\ref{eq:equi}) holds for any {\it well-rounded\/} family
of sets $S_T\subset V$ in place of $V\cap B_T$.  Recall that a family
$\{S_T\}$ is called well-rounded if for any $\e>0$ there exists a
neighborhood $\mathcal{O}$ of $e$ in $G$ such that
\begin{equation} \label{eq:well-rounded}
  \frac{\vol(\mathcal{O}\cdot \partial S_T)}{\vol(S_T)}<\e
\end{equation}
for all sufficiently large $T>0$. For any 
$I\subset \Delta_\sigma$, $w\in \cW$ and $\Omega\subset KM_I/(M_I\cap
wHw\inv)$, we consider a family of sets
\begin{equation} \label{eq:S_T-Omega}
S_T(\Omega,w)=\tilde \Omega A_I^+wv_0\cap B_T,
\end{equation}
where $\tilde\Omega\subset KM_I$ is such that
$\Omega=\tilde\Omega(M_I\cap wHw\inv)$; the set $S_T(\Omega,w)$ is
well defined because $mawv_0=awv_0$ for all $a\in A_I$ and $m\in
(M_I\cap wHw\inv)$. 

Using the strong wavefront lemma, and the volume computation in
\cite{gos} (cf.~Proposition~\ref{p:volume}) we obtain the following:
\begin{thm}\label{cor:counting}
  For every $I\subset \Delta_\sigma$, $w\in\cW$, and a bounded
  measurable set $\Omega\subset KM_I/(M_I\cap wHw\inv)$ with positive
  measure and boundary of measure zero\footnote{The measure on
  $KM_I/(M_I\cap wHw\inv)$ is understood in terms of the
  identification \eqref{eq:mea} and \eqref{eq:dm}.}, the family
  $\{S_T(\Omega,w)\}_{T\to\infty}$ is well-rounded. In particular,
  \begin{align*}
    \# (\Gamma v_0\cap S_T(\Omega,w)) \sim_{T\to \infty}
    \operatorname{Vol}(S_T(\Omega,w))\sim_{T\to\infty}
    C_I(\Omega,w)\cdot T^{a_I}(\log T)^{b_I-1},
  \end{align*}
where $a_I\in\mathbb{Q}^+$, $b_I\in\mathbb{N}$, and $C_I(\Omega,w)>0$.
\end{thm}

We will give explicit formulas for $a_I$, $b_I$, and $C_I(\Omega,w)$
in section~\ref{subsec:volume_estimate}.  In particular,
$C_I(\Omega,w)$ can be computed using a $G$-invariant measure
supported on one of the components of the Satake boundary of $V$.

\begin{rem} {\rm \begin{enumerate} \item Although a similar counting
      question was considered in \cite{gos}, the sets $S_T(\Omega,w)$
      do not fit into the framework of \cite{gos}.  For the space of
      quadratic forms $\mathcal{Q}_W$, the counting results in
      \cite{gos} are always of order $T^{(\dim W)(\dim W-1)/2}$ (see
      \cite[Section~2.3]{gos}). On the other hand,
      Theorem~\ref{th:quad} exhibits different asymptotic behaviors
      depending on the choice of the decomposition \eqref{eq:W}.

\item 
In order to deduce Theorem~\ref{cor:counting} from 
Theorem~\ref{th:sw2}, which applies only to $(J,c)$-regular elements,
we show that the set of non-$(J,c)$-regular elements in
$S_T(\Omega,w)$ has negligible volume compared to the volume of
$S_T(\Omega,w)$ for sufficiently small values of $c$. 

% We mention the following issue that requires extra care in the
%   proof of Theorem~\ref{cor:counting}: usually,
%   \[
%   \liminf_{T\to\infty} 
% \frac{\vol(S_T(\Omega)\cap 
% \{\text{non-$(J,c)$-regular elements in $V$}\})}{\vol(S_T(\Omega)}>0
%   \]
%   for all $c>0$. Nonetheless, we show that
%   \[
%   \limsup_{T\to\infty} \frac{\vol(S_T(\Omega)\cap
%     \{\hbox{non-$(J,c)$-regular elements in
%       $V$}\})}{\vol(S_T(\Omega))}\ll c,
%   \]

\end{enumerate}
}
\end{rem}

\subsection{Acknowledgment.} We would like to thank Yves Benoist
for useful comments.

\section{Strong wavefront lemma}\label{sec:wave}

This section is devoted to the proofs of Theorems~\ref{th:sw} and
\ref{th:sw2}. We use the same notation as in the introduction. Since
any two Riemannian metrics are bi-Lipschitz in a neighborhood of
identity, it suffices to prove the theorems for one such metric. It
will be convenient to work with the right-invariant Riemannian metric
$d$ induced by the positive definite form
\begin{equation*}
  B(X, Y)=-\operatorname{Tr}(\operatorname{ad} X \circ \operatorname{ad}
  (\theta(Y)),\quad X,Y\in\mathfrak{g}.
\end{equation*}
We will use the following properties of $B$:
\begin{align*}
  &B(\mathfrak{g}_\alpha,\mathfrak{g}_\beta)=0\quad\hbox{for all $\alpha\ne\beta\in \Sigma_\sigma\cup\{0\}$},\\
  &B^\theta=B^\sigma=B.
\end{align*}

\begin{rem} \label{rem:Weyl} \rm In many of the results stated in the
  introduction, we fix $w\in\cW$ representing a Weyl group
  element. The explanation given below shows that for proofs, we can
  assume that $w=e$ and have simpler notation. 

  Let $i_w$ denote the inner conjugation on $G$ by $w$; that is,
  $i_w(g)=wgw\inv$ for all $g\in G$. Then
  $\sigma_w:=i_w\circ\sigma\circ i_w\inv$ is also an involution of $G$
  and $wHw\inv$ is the associated symmetric subgroup. Note that
  $\sigma_w(a)=a\inv$ for any $a\in A$. Also $\theta\circ
  \sigma_w=\sigma_w\circ \theta$. Therefore in order to prove some of
  the results stated in the introduction for a fixed $w\in \cW$, we can
  replace $\sigma$ by $\sigma_w$, $H$ by $wHw\inv$, and $v_0$ by
  $wv_0$, and assume that $w=e$.
  \end{rem}

For $\varepsilon>0$ and $S\subset G$, we set
$$
S_\varepsilon=\{s\in S: d(s,e)<\varepsilon\}.
$$

For $I\subset \Delta_\sigma$ and $c>0$, we define $$ A_I^+(c)=\{a\in
A^+:\, \beta(\log a)\ge c \hbox{ if $\beta\in \Delta_\sigma-I$ and }
\beta(\log a)<c\hbox { if $\beta\in I$} \}.
$$
For instance, if $I=\{\beta\}$, then $A_{I}^+(c)$ forms a system of
neighborhoods of the wall $\{a\in A^+: \beta (\log a)=0\}$ in $A^+$.
% We also set $\la{a}_I=\ker(I)\subset\la{a}$ and denote by $Z_I$ the
% centralizer of $\la{a}_I$ in $G$.

\begin{thm}\label{lm:swl}
  For $I\subset \Delta_\sigma$ and $c>0$, there exist $\e_0>0$ and
  $\sigma>1$ such that for every $0<\e<\e_0$ and $a\in A_I^+(c)$,
$$
G_\e\cdot a\subset K_{\sigma\e}\cdot Z_{I,\sigma\e}\cdot a \cdot
H_{\sigma\e}.
$$
\end{thm}

We consider the Lie subalgebra
\[
\la{n}^+_I=\bigoplus_{\beta\in \Sigma_\sigma^+:\,
  \beta|_{\la{a}_I}\neq 0} \la{g}_\beta \quad \text{and} \quad
\la{n}^-_I=\bigoplus_{\beta\in \Sigma_\sigma^+:\,
  \beta|_{\la{a}_I}\neq 0} \la{g}_{-\beta},
\]
and the corresponding analytic subgroups $N^+_I$ and $N^-_I$. Note
that the Lie algebra of $Z_I$ is given by
$$
\la{z}_I=\bigoplus_{\beta\in \Sigma_\sigma\cup\{0\}:,\,
  \beta|_{\la{a}_I}= 0} \la{g}_\beta,
$$
and we have the decomposition
\begin{equation} \label{eq:i} \la{g}=\la{n}^-_I\oplus \la{z}_I\oplus
  \la{n}^+_I.
\end{equation}

\begin{Lem}\label{eq:decomp1}
  There exist $\theta>1$ and $\e_0>0$ such that for every $0<\e<\e_0$,
  \begin{align*}
    G_\e\subset N_{I,\theta\e}^-Z_{I,\theta\e}
    H_{\theta\e}\quad\hbox{and}\quad G_\e\subset K_{\theta\e}
    Z_{I,\theta\e} N^+_{I,\theta\e}.
  \end{align*}
\end{Lem}

\begin{proof}
  Since $\sigma|_\la{a}=-id$, we have $\sigma(\mathfrak
  n^-_{I})\subset \mathfrak n^+_{I}$, and for every $x\in
  \mathfrak{n}^{+ }_{I}$,
$$ 
x=(x+\sigma(x))-\sigma(x)\in \mathfrak{h}+\mathfrak{n}^{-}_{I}.
$$
Hence, it follows from (\ref{eq:i}) that
$$
\mathfrak{g}=\mathfrak{n}^-_I + \la{z}_I+\mathfrak{h}.
$$
Since $\mathfrak n_I^-\cap \mathfrak h=0$, there exists a subspace
$\mathfrak z_0$ of $\mathfrak z_I$ such that
$$
\mathfrak g =\mathfrak n_I^- \oplus \mathfrak z_0 \oplus \mathfrak h.
$$
Then the product map $N_I^-\times \exp (\mathfrak z_0) \times H\to G$
is a diffeomorphism at a neighborhood of the identity.  In particular,
it is bi-Lipschitz, and the first claim follows.  The proof of the
second claim is similar.
\end{proof}

\begin{lem}\label{l_contract}
  For $I\subset\Delta_\sigma$ and $c>0$, there exist $\varepsilon_0>0$
  and $\alpha\in (0,1)$ such that for every $0< \varepsilon
  <\varepsilon_0$ and $a\in A_I^+(c)$,
$$
a^{-1}N_{I,\varepsilon}^{+} a\subset
N_{I,\alpha\varepsilon}^{+}\quad\hbox{and}\quad aN_{I,\varepsilon}^{-}
a^{-1}\subset N_{I,\alpha\varepsilon}^{-}.
$$
\end{lem}
\begin{proof}
  For
$$
X=\sum_{\beta\in \Sigma_\sigma^+, \beta|_{\mathfrak a_{I}}\ne 0}
X_{\beta}\in \mathfrak n_{I}^+,\quad X_\beta\in \mathfrak g_\beta,
$$
we have
$$
\operatorname{Ad}(a^{-1})X=\sum_\beta \operatorname{Ad}(a^{-1})
X_\beta=\sum_\beta e^{-\beta(\log a)}X_\beta.
$$
Note that if $\beta=\sum_{\alpha\in \Delta_\sigma} n_\alpha \alpha\in
\Sigma_\sigma^+$ with $n_\alpha\ge 0$ satisfies
$\beta|_{\mathfrak{a}_{I}}\ne 0$, then $n_{\alpha}\ge 1$ for some
$\alpha \in \Delta_\sigma-I$.  Hence, for $a\in A^+_I(c)$, we have
$\beta(\log a) \ge c$ and
$$
\|\operatorname{Ad}(a^{-1})X_\beta\|\le e^{-c}\|X_\beta\|.
$$
Since the root spaces $\mathfrak g_{\beta}$ are orthogonal to each
other,
\begin{equation}\label{eq:c1}
  \|\operatorname{Ad}(a^{-1})X\|\le e^{-c}\|X\|.
\end{equation}
Since the differential of the exponential map $\exp: {\mathfrak
  n}^{+}_I \to N^{+}_I$ is identity at $0$, we can find a small ball
$U$ at $0$ in $\mathfrak{n}_I^+$ such that for every $Y\in U$.
\begin{equation}\label{eq:c2}
  e^{-c/3}\|Y\|\le d(\exp(Y),e)\le e^{c/3}\|Y\|.
\end{equation}
Note that for $a\in A^+$, we have $\operatorname{Ad}(a^{-1})U\subset
U$.  Combining (\ref{eq:c1}) and (\ref{eq:c2}), we deduce that for
$a\in A^+_I(c)$ and $n=\exp(X)\in \exp(U)$,
\begin{align*}
  d(a^{-1}na,e)&=d(\exp(\operatorname{Ad}(a^{-1})X),e)\le e^{c/3}
  \|\hbox{Ad}(a^{-1})X\|\\
  &\le e^{-2c/3} \|X\| \le e^{-c/3}d(n,e).
\end{align*}
This proves the claim for $N^+_I$. The claim for $N_I^-$ is proved
similarly.
\end{proof}

\begin{lem}\label{eq:ll}
  For $I\subset \Delta_\sigma$ and $\tau>1$, there exists $\e_0>0$
  such that for every $z\in Z_{I,\e_0}$ and $0<\e < \e_0$,
  \begin{equation*}
    z N^{+}_{\e} z^{-1}\subset N^{+}_{\tau \e}\quad\hbox{and}\quad
    z N^{-}_{\e} z^{-1}\subset N^{-}_{\tau \e}.
  \end{equation*}
\end{lem}

\begin{proof}
  It is easy to check that $L_I$ normalizes $N_I^\pm$.

  We can choose $\e_0>0$ so that
  \begin{align*}
    &\|\hbox{Ad}(z)X\|\le \tau^{1/3}\|X\|, & z\in
    Z_{I,\e_0},\; X\in \la{n}^+_I,\\
    \tau^{-1/3}\|X\|\le &d(\exp(X),e)\le \tau^{1/3}\|X\|, & X\in
    \hbox{Ad}(Z_{I,\e_0})\exp^{-1}(N^+_{I,\e_0}).
  \end{align*}
  Then for every $n=\exp(X)\in N^+_{I,\e_0}$,
  \begin{align*}
    d(znz^{-1},e)&=d(\exp(\hbox{Ad}(z)X),e)\le \tau^{1/3}\|\hbox{Ad}(z)X\|\\
    &\le \tau^{2/3}\|X\|\le \tau d(n,e).
  \end{align*}
  This proves the first part of the lemma. The proof of the second
  part is similar.
\end{proof}

\begin{lem}\label{l:pm}
  For $I\subset \Delta_\sigma$ and $\gamma>1$, there exists
  $\varepsilon_0>0$ such that for every $0<\varepsilon
  <\varepsilon_0$,
$$
N^+_{I,\varepsilon}\subset N^-_{I,\gamma\varepsilon}Z_{I,
  \varepsilon}H_{2\gamma \varepsilon} \quad\hbox{and}\quad
N^-_{I,\varepsilon}\subset K_{2\gamma
  \varepsilon}Z_{I,\varepsilon}N^+_{I,\gamma\varepsilon}.
$$
\end{lem}

\begin{proof}
  As in the proof of Lemma~\ref{eq:decomp1}, we choose a subspace
  $\mathfrak z_0$ of $\mathfrak z_I$ such that the product map
  $N_I^-\times \exp (\mathfrak z_0) \times H\to G$ is a diffeomorphism
  in a neighborhood of the identity.  Denote by $f$ the local inverse
  the product map:
$$
f=(f_1,f_2,f_3): U\to N_I^-\times\exp (\mathfrak z_0)\times H
$$
where $U$ is a neighborhood of identity in $G$.  For $X\in \mathfrak
n_I^+$, the derivative $(df)_e$ is given by $$(df)_e(X)= (-\sigma(X),
0, X+\sigma(X)) \in \mathfrak n_I^-\oplus \mathfrak z_0\oplus
\mathfrak h.$$ Since the Riemannian metric at identity is invariant
under $\sigma$, we have for $X\in\mathfrak{n}^+_I$,
$$
\|(df_1)_e (X)\|=\|X\|,\quad (df_2)_e =0,\quad \|(df_3)_e (X)\|\le
2\|X\|.
$$
This implies that for sufficiently small $\e>0$,
$$
f(N^+_{I,\varepsilon})\subset N^-_{I,\gamma\varepsilon}\times Z_{I,
  \varepsilon}\times H_{2\gamma \varepsilon}.
$$
This proves the first claim. The proof of the second claim is similar.
\end{proof}

\begin{lem}\label{l_shah}
  For $I\subset \Delta_\sigma$ and $c>0$, there exist $0< \beta<1$ and
  $\varepsilon_0>0$ such that for every $0< \varepsilon,\delta
  <\varepsilon_0$ and $a\in A_I^+(c)$,
 $$ K_\varepsilon Z_{I,\e} a Z_{I,\e} N_{I,\delta}^+
 H_\varepsilon\subset K_{\varepsilon+4\delta} Z_{I,\e+4\delta} a
 Z_{I,\e+4\delta} {N}_{I,\beta \delta}^+ H_{\e+4\delta}.
$$
\end{lem}

\begin{proof}
  For simplicity, we write $N^{\pm}_I=N^{\pm}$ and $Z_{I}=Z$.

  Choose $\alpha=\alpha (c)\in (0,1)$ as in Lemma~\ref{l_contract},
  $\gamma\in (1,2)$ so that $\alpha \gamma^2 <1$, and $\tau>1$ so that
  $\tau^5\alpha \gamma^2 <1$.  Let $\varepsilon_0>0$ be such that
  Lemma~\ref{l_contract}, Lemma~\ref{eq:ll}, and Lemma~\ref{l:pm}
  hold.
  % Let $\e_1>0$ be such that
%$$2(D \e_1)<\e_0 \cdot (1-\alpha^{-1})$.
  Fixing $0<\varepsilon <\varepsilon_0$, let $k_0\in K_\varepsilon$,
  $x_0, y_0\in Z_\varepsilon$, $n_0^+\in {N}_\delta^+$, and $h_0\in
  H_\varepsilon$.  Then \begin{align*}
    & k_0x_0 a y_0 n_0^+ h_0 \\
    =&k_0x_0 a y_0(n^-_1 y_1h_1)h_0 &\text{by Lemma~\ref{l:pm}}\\
    &&\hbox{with $n^-_1\in N^-_{\gamma\delta}, y_1\in Z_{ \delta},
      h_1\in H_{2 \gamma \delta}$} \\
    =&k_0 n^-_2 x_0a y_0y_1h_1h_0
    &\text{by Lemma~\ref{eq:ll} and Lemma~\ref{l_contract}}\\
    && \hbox{with $n^-_2\in N^-_{\tau^2 \alpha \gamma\delta}$} \\
    =&k_0(k_2{x_2}n_2^+)x_0a y_0y_1 h_1h_0  &\text{by Lemma~\ref{l:pm},}\\
    & &\text{with $k_2\in K_{2 \tau^2 \alpha\gamma^2 \delta}, {x}_2\in
      {Z}_{\tau^2 \alpha \gamma \delta}, {n}^+_2\in
      {N}^+_{\tau^2\alpha\gamma^2\delta}$
    }\\
    =& k_0k_2 (x_2 x_0ay_0y_1)n^+_3 h_1h_0
    &\text{by Lemma~\ref{eq:ll} and Lemma~\ref{l_contract}}\\
    && \hbox{with $n^+_3\in N^{+}_{\tau^5 \alpha^2 \gamma^2 \delta}$.}
  \end{align*}
  Since $\tau^5\alpha \gamma^2 <1$, we have
 $$k_0k_2\in K_{\e+ 4 \delta},\;\;
 x_2x_0, y_0y_1 \in Z_{\varepsilon +4 \delta},\;\; n^+_3\in
 N^{+}_{\beta \delta},\;\;h_1h_0\in H_{\e+ 4 \delta}. $$ where
 $\beta=\tau^5 \alpha^2\gamma^2<1$.
\end{proof}

\begin{proof}[Proof of Theorem~\ref{lm:swl}]
  Set $N^{\pm}_I=N^{\pm}$ and $Z_{I}=Z$ for simplicity. In view of
  Remark~\ref{rem:Weyl} without loss of generality we may assume that
  $w=e$.

  We choose $\e_0>0$ so that Lemma~\ref{eq:decomp1} (for some
  $\theta>1$), Lemma~\ref{l_contract}, and Lemma~\ref{l_shah}
  hold. Because of Lemma~\ref{eq:decomp1}, it suffices to show that
  $$
  K_\varepsilon Z_{\e} {N}^+_\varepsilon\cdot a\subset K_{\sigma\e}
  (Z_{\sigma\e}a) H_{\sigma\e}
  $$
  for some $\sigma>1$.  Also by Lemma~\ref{l_contract},
  $$
  K_\varepsilon Z_{\e} {N}^+_\varepsilon\cdot a\subset K_\varepsilon
  (Z_{\e} a Z_\e) {N}^+_\varepsilon H_\e.
  $$
  Now we can apply Lemma~\ref{l_shah} inductively.  We consider $\e>0$
  such that
  \begin{equation}\label{eq:e}
    \varepsilon+\frac{4\varepsilon}{1-\beta}<\varepsilon_0.
  \end{equation}
  Setting $\varepsilon_0=\delta_0=\varepsilon$, we apply
  Lemma~\ref{l_shah} to find
  $$ \varepsilon_{i+1}< \varepsilon_{i}+2\delta_{i},\;\; \delta_{i+1}< \beta\delta_{i}$$
  such that for every $a\in A_I^+(c)$,
  \begin{equation*}
  K_{\varepsilon_i} Z_{\e_i} a Z_{\e_i} {N}^+_{\delta_i}
  H_{\varepsilon_i}\subset
  K_{\varepsilon_{i+1}}Z_{\e_{i+1}} a Z_{\e_{i+1}}
  {N}^+_{\delta_{i+1}} H_{\varepsilon_{i+1}}.
\end{equation*}
Note that
\begin{align*}
  \delta_i &<\varepsilon\beta^i\quad\hbox{and}\quad
  \varepsilon_i<\varepsilon +4\varepsilon\frac{1-\beta^{i}}{1-\beta}.
\end{align*}
Hence by (\ref{eq:e}), $\e_i,\delta_i<\e_0$, and we can continue this
process indefinitely.

It follows that for every $g\in K_\varepsilon (Z_{\e} a Z_\e)
{N}^+_\varepsilon H_\e$, there exist sequences $k_i\in K_{\e_{i}},
x_i, y_i\in Z_{\e_i}, n_i\in N^+_{\delta_i}, h_i\in H_{\e_i}$ such
that $g=k_ix_i a y_i n_i h_i$ for all $i\ge 1$.  Since $\delta_i\to
0$, $n_i\to e$.  Also, passing to a subsequence, we may assume that
$k_i\to k, x_i\to x, y_i\to y, h_i\in h$.  Then
\begin{equation*}
  g=kxayh \subset K_{\rho\e}Z_{\rho\e} a Z_{\rho\e} H_{\rho\e}
\end{equation*}
with $\rho=1+4(1-\beta)^{-1}$. We have decomposition $a=a_1a_2$ where
$a_1\in A_I^+$ and $a_2$ is in the fixed compact set determined by
$c$.  This implies that for some $\tau>1$,
$$
aZ_{\rho \e}a^{-1}\subset Z_{\tau \rho \e},
$$
and the theorem follows.
\end{proof}

\begin{proof}[Proof of Theorem~\ref{th:sw2}] 
  There exists $\zeta>1$ such that $k^{-1}\mathcal{O}_\e
  k\subset\mathcal{O}_{\zeta\e}$ for every $k\in K$. Then for
  $g=kawh\in KA^+\cW H$, we have
$$
\mathcal{O}_\e\cdot g\subset k(\mathcal{O}_{\zeta\e}a)wh.
$$
Due to Remark~\ref{rem:Weyl}, without loss of generality, we may
assume that $w=e$.

Since $M_{I_1}\subset M_{I_2}$ for $I_1\subset I_2$, we may assume
that $J$ is maximal such that $a$ is $(J,c)$-regular. Then $a\in
A_I^+(c)$.  We have the decomposition
\begin{equation}\label{eq:z}
  \la{z}_I=(\la{z}_I\cap \la{k})\oplus (\la{m}_I\cap \la{p}\cap
  \la{q})\oplus \la{a}_I \oplus (\la{z}_I\cap\la{h}) 
\end{equation}
(see \cite[equation (4.24)]{gos}).  Hence, the product map
$$
(Z_I\cap K)\times \exp(\la{m}_I\cap \la{p}\cap\la{q}) \times A_I\times
(Z_I\cap H)\to Z_I
$$
is a diffeomorphism in a neighborhood of identity, and there exists
$\eta>1$ such that for sufficiently small $\e>0$,
$$
Z_{I,\e}\subset (Z_I\cap K)_{\eta\e} \exp(\la{m}_I\cap \la{p}\cap
\la{q})_{\eta\e} A_{I,\eta\e}(Z_I\cap H)_{\eta\e}.
$$
Therefore, it follows from Theorem~\ref{lm:swl} that
$$
\mathcal{O}_\e\cdot a\subset K_{\sigma\e}Z_{\sigma\e} a H_{\sigma
  \e}\subset
K_{(\sigma+\sigma\eta)\e}M_{I,\sigma\eta\e}(A_{I,\sigma\eta \e}
a)H_{(\sigma+\sigma\eta)\e}.
$$
This proves the theorem.
\end{proof}

\begin{proof}[Proof of Theorem~\ref{th:sw}]
  Suppose that in Theorem~\ref{th:sw2} we have $J=\Delta_\sigma$.
  Then $Z=C_G(A)$ is $\sigma$- and $\theta$-invariant, and
$$
\la{z}=(\la{z}\cap \la{k})\oplus (\la{z}\cap \la{p}\cap \la{q}) \oplus
(\la{z}\cap\la{h}).
$$
Since $\la{a}$ is a maximal abelian subspace of $\la{p}\cap\la{q}$,
$\la{z}\cap \la{p}\cap \la{q}=\la{a}$.  Hence, decomposition
(\ref{eq:z}) becomes
$$
\la{z}=(\la{z}\cap \la{k})\oplus \la{a} \oplus (\la{z}\cap\la{h}),
$$
and we complete the proof as in Theorem~\ref{th:sw2}.
\end{proof}

\section{Well-roundedness of sectors $S_T(\Omega,w)$}
First we need a precise description of the
measure on the set 
$$
KM_I(wv_0)\cong KM_I/(M_I\cap wHw\inv).
$$ 

\subsection{Description of a measure on $KM_I/(M_I\cap wHw\inv)$} 
Fix $w\in \cW$. Let $\sigma_w=i_w\circ \sigma \circ i_w\inv$ be the
involution as in Remark~\ref{rem:Weyl}. Then
$\sigma_w\circ\theta=\theta\circ\sigma_w$.  Also the semisimple group $M_I$ is stable
under $\sigma_w$ and $\theta$, and hence $M_I$ admits the generalized
Cartan decomposition (see~\cite[Proposition~4.22]{gos}):
\begin{equation} \label{eq:M_I-Cartan}
  M_I=(M_I\cap K)A^I(M_I\cap wHw\inv)=(M_I\cap K) A^{I,+}\cW_I(M_I\cap wHw\inv),
\end{equation}
where $A^I$ is the orthogonal complement of $A_I$ in $A$ and it is the
Cartan subalgebra of $M_I$ associated to the symmetric pair $(M_I\cap
K,M_I\cap wHw\inv)$, and $A^{I,+}=\{a\in A^I:\alpha(\log a)\geq
0,\,\forall \alpha\in I\}$ is a positive Weyl chamber; and
$\cW_I\subset M_I$ is a set of representatives of the associated Weyl
group, which is generated by the reflections $\{s_\alpha\}_{\alpha\in
  I}$. An invariant measure, say $\lambda$ on $M_I/(M_I\cap wHw\inv)$
is given as follows: for any $f\in C_c(M_I/M_I\cap wHw\inv)$,
\[
\int f d\lambda = \sum_{w_1\in \cW_I}\int_{K\cap M_I}dk \int_{A^{I,+}}
f(kaw_1(M_I\cap wHw\inv))\delta_I(a)\,da
\]
where
\[
\delta_I(a)=\prod_{\alpha\in \Sigma_\sigma^+\cap\langle I\rangle}
(\sinh\alpha(a))^{l_\alpha^+}(\cosh(\alpha))^{l_\alpha^-},
\]
and $l_\alpha^\pm$ denote the dimensions of the $(\pm 1)$-eigenspaces
of $\sigma\theta$ on $\la{g}^\alpha$.
 
Therefore we can identify
\begin{equation}\label{eq:mea}
  KM_I/(M_I\cap wHw\inv )\cong K\times A^{I,+}\times \cW_I,
\end{equation}
and treat $KM_I/(M_I\cap wHw\inv)$ as a product measure space.

On the other hand, once we fix a measurable section $s_1:K/(K\cap
M_I)\to K$ for the natural quotient map, we can identify $K\times
A^{I,+}\times \cW_I$ with $K/(K\cap M_I)\times M_I/(M_I\times
wHw\inv)$.  We consider the measure on $K\times A^{I,+}\times \cW_I$
such that it corresponds to the product of the invariant measures on
the product space $K/(K\cap M_I)\times M_I/(M_I\cap wHw\inv)$, where
the Haar measures on $K$ and $K\cap M_I$ are normalized.  This
measure, in view of \eqref{eq:mea}, will give rise to the integral
$d\bar m$ on $KM_I/(M_I\cap wHw\inv)$ given as follows: for any $f\in
C_c(KM_I/M_I\cap wHw\inv)$,
\begin{equation} \label{eq:dm} \int f(\bar m)d\bar m := \sum_{w_1\in
    \cW_I}\int_{K}dk \int_{A^{I,+}}f(kaw_1(M_I\cap
  wHw\inv))\delta_I(a)\,da.
\end{equation}

\subsection{Volume estimate for the sectors $S_T(\Omega,w)$}
\label{subsec:volume_estimate}
Let $\lambda_\iota$ denote the highest weight for the irreducible
representation $\iota$. We express
\begin{equation} \label{eq:lambda_i}
\lambda_\iota=\sum_{\alpha\in\Delta_\sigma} m_\alpha\alpha
\end{equation}
and the sum of positive roots (with multiplicities)
\begin{equation} \label{eq:rho}
2\rho=\sum_{\alpha\in\Delta_\sigma} u_\alpha\alpha.
\end{equation}

Let $I\subset \Delta_\sigma$. Set
\begin{align}
  a_I&=\max\{\frac{u_\alpha}{m_\alpha}:\, \alpha\in\Delta_\sigma-I\}, 
  \label{eq:a_I}\\
  b_I&=\#\{\alpha\in\Delta_\sigma-I:\, \frac{u_\alpha}{m_\alpha}=a_I\}
  \label{eq:b_I}.
\end{align}

\begin{Prop}\label{p:volume}
  For any $w\in \cW$ and a bounded measurable set $\Omega\subset K
  M_I/(M_I\cap wHw\inv)$ with positive measure and zero boundary
  measure, there exists $C_I(\Omega,w)>0$ such that
  $$ \vol(S_T(\Omega,w))\sim_{T\to\infty} C_I(\Omega,w)\cdot
  T^{a_I}(\log T)^{b_I-1}.
  $$
\end{Prop}

\begin{proof}
  From \cite[Theorem~2.5]{hs} (see also \cite{gos}) one deduces that
  a $G$-invariant measure on $G/H$ is given by
  \begin{equation}\label{eq:volume}
    \int_{G/H} f\, d\mu =  
    \sum_{w\in\mathcal{W}}\int_{\bar m\in KM_I/(M_I\cap{w}Hw^{-1})} 
    \int_{a\in A_I^{+}} 
    f(\bar m a w H)\xi_I(a)\,da d\bar m,\quad f\in C_c(G/H),
  \end{equation}
  where $da$ denotes a Haar measure on $A_I$, and $d\bar m$ is
  described in the paragraph following \eqref{eq:mea}, and
  \begin{equation}\label{eq:xi}
    \xi_I(a)=\prod_{\alpha\in\Sigma_\sigma^+-\langle I\rangle}
    \sin(\alpha(\log a))^{l_\alpha^+}\cos(\alpha(\log a))^{l_\alpha^-}.
  \end{equation}
  Here $l_\alpha^\pm$ denote the dimensions of the $(\pm
  1)$-eigenspaces of $\sigma\theta$ in $\la{g}_\alpha$.  We decompose
  $\xi_I$ as a linear combination of functions $\exp(\chi(a))$ where
  $\chi$'s are characters of $A_I$. Note that $2\rho$ is the maximal
  character in this decomposition. In view of
  equations~\eqref{eq:lambda_i},\eqref{eq:rho}, and \eqref{eq:a_I}, we
  define
$$
I_0=I\cup \{\alpha\in\Delta_\sigma-I:\, \frac{u_\alpha}{m_\alpha}<a_I\}.
%\quad\hbox{and}\quad I_0=I\cup J_0.
$$ By the computation using \cite[Theorem 6.1]{gos}, as done in the
proof of \cite[Theorem 6.4]{gos}, applied to $\la{a}_I$ in place of
$\la{a}$, there exists a locally finite measure $\eta_{I,w}$ on $W$
such that for every $f\in C_c(W)$,
\begin{equation}\label{eq:limit}
  \lim_{T\to\infty}\frac{1}{T^{a_I}(\log T)^{b_I-1}}
  \int_{a\in A_I^+} f(awv_0/T)\xi_I(a)\,da = \int_W f\, d\eta_{I,w},
\end{equation}
where the measure $\eta_{I,w}$ can be described as follows: 
\begin{equation} \label{eq:A_I}
\int_W f\,d\eta_{I,w} =\int_{\bar b\in D^+} f(b (wv_0)^{I_0})\,
\tilde \xi_I(b)\,d\bar b,
\end{equation}
where $D^+=\exp{\la{d}^+}$, 
\[
\la{d}^+=\{\bar b\in \la{a}_I/(\la{a}_{I_0}\cap \ker\rho):\alpha(b)\geq
0,\,\forall\alpha\in I_0\}, 
\]
$d\bar b$ denotes the Haar measure on $A_I/(A_I\cap\exp(\ker\rho))$,
$v_0^{I_0}$ is the projection of $v_0$ to the sum of the weight spaces
with weights of the form $\lambda_\iota-\sum_{\alpha\in I_0}
m_\alpha\alpha$, $m_\alpha\ge 0$, and
\begin{equation} \label{eq:xi-tilde}
\tilde\xi_I(b)=\left( \prod_{\alpha\in(\Sigma_\sigma^+\cap \langle
    I_0\rangle) -\langle I\rangle} \sin(\alpha(\log
  b))^{l_\alpha^+}\cos(\alpha(\log b))^{l_\alpha^-}\right)
\cdot\exp\left(\sum_{\alpha\in \Sigma_\sigma^+ - \langle I_0\rangle} 
u_\alpha\alpha(\log b) \right).
\end{equation}
Moreover it follows from \eqref{eq:limit} that $\eta_{I,w}$ is a
   homogeneous measure of degree $a_I$.

  Fix any $m\in KM$. Let $c>1$ and take a continuous function
  $\psi: [0,\infty]\to [0,1]$ such that $\supp(\psi)\subset [0,c]$ and
  $\psi=1$ on $[0,1]$.  Setting $f(y)=\psi(\|my\|)$, we have
 \begin{equation} \label{eq:A_I-B_T}
  \int_{A^+_I} \chi_{B_T}(mav_0)\xi_I(a)da \le \int_{A^+_I}
  f(awv_0/T)\xi_I(a)da.
 \end{equation}
Now by \eqref{eq:limit} and \eqref{eq:A_I-B_T},
$$
\limsup_{T\to\infty}\frac{1}{T^{a_I}(\log T)^{b_I-1}} \int_{A_I^+}
\chi_{B_T}(mawv_0)\xi_I(a)da \le \int_W f\, d\eta_{I,w}\le c^{a_I}
\eta_{I,w}(m\inv B_1).
$$
The lower estimate for $\liminf$ is proved similarly.

Hence, taking $c\to 1^+$, we obtain
\begin{equation} \label{eq:A_I-plus}
\lim_{T\to\infty}\frac{1}{T^{a_I}(\log T)^{b_I-1}} \int_{A_I^+}
\chi_{B_T}(mawv_0)\xi_I(a)da = \eta_{I,w}(m\inv B_1).
\end{equation}

In view of \eqref{eq:mea} let $s:KM_I/(M_I\cap wHw\inv)\to KM_I$
denote the measurable section of the obvious quotient map. 
Since
$$
S_T(\Omega,w)=\Omega A_I^+wv_0\cap B_T,
$$
\begin{equation}
  \label{eq:S_T}
  \hbox{Vol}(S_T(\Omega,w))=\int_{\bar m\in \Omega} \int_{a\in A_I^+}
  \chi_{B_T}(s(\bar m)awv_0)\xi(a)\,da d\bar m. 
\end{equation}
Therefore  from \eqref{eq:A_I-plus}, using the dominated convergence
theorem, we deduce that
\begin{align}\label{eq:asympt}
  C_I(\Omega,w):=\lim_{T\to\infty}\frac{\vol(S_T(\Omega,w))}{T^{a_I}(\log
    T)^{b_I-1}} = \int_{\bar m\in \Omega} \eta_{I,w}(s(\bar
  m)^{-1}B_1)d\bar m.
\end{align}
Note that there exists $\delta>0$ such that $s(\bar m)^{-1} B_1\supset
B_{\delta}$ for all $\bar m\in \Omega$, and because $\eta_{I,w_0}$ is
homogeneous, $\eta_{I,w}(B_\delta)>0$.  Hence $C_I(\Omega,w)>0$.
\end{proof}

\begin{rem} 
%\begin{Prop} \label{p:C_I}
The value of the parameter $C_I(\Omega,w)$ in  the statement of
Proposition~\ref{p:volume} is given by
\begin{equation} \label{eq:C_I-Omega}
C_I(\Omega,w)=\nu_{I_0,w}(B_1\cap \Omega D^+(wv_0)^{I_0}),
\end{equation}
where $\nu_{I_0,w}$ is a $G$-invariant measure on the $G$-orbit $G(wv_0)^{I_0}$.
%\end{Prop}

\bigskip
%\begin{proof} 
\rm
This formula can be justified as follows: combining \eqref{eq:limit},
  \eqref{eq:A_I}, \eqref{eq:A_I-plus}, \eqref{eq:asympt} and
  \eqref{eq:dm} we get
\begin{align} 
  C_I(\Omega,w) & =\int_{\bar m\in\Omega} d\bar m \int_{\bar b\in D^+}
  \chi_{B_1}(\bar m\bar b(wv_0)^{I_0})\xi_I(b)\,d\bar b \nonumber\\
  & =\int_{k\in K}dk\int_{a\in A^{I,+}}\int_{\bar b\in
    D^+}\chi_{\Omega}(ka)\chi_{B_1}(kab(wv_0)^{I_0})\delta_I(a)\tilde\xi_I(b)\,
  da d\bar b, \label{eq:C_I}
\end{align}
where 
\begin{equation} \label{eq:deta_I}
\delta_I(a)=\prod_{\alpha\in(\Sigma_\sigma^+\cap \langle
    I\rangle)} \sin(\alpha(\log
  a))^{l_\alpha^+}\cos(\alpha(\log a))^{l_\alpha^-},
\end{equation}
$l_\alpha^\pm$ are the dimensions of the $(\pm 1)$-eigenspaces of
$\sigma\theta$ acting on $\la{g}^\alpha$. 

Since
$$
\la{a}_{I_0}\cap \ker\rho=\la{a}_{I_0}\cap \ker\lambda_\iota,
$$
it follows from \cite[Theorem~5.1]{gos} that the orbit $G(wv_0)^{I_0}$
supports a $G$-invariant measure $\nu_{I_0}$. Now comparing the
formula \eqref{eq:C_I} with the formula (5.3) in \cite[Theorem~5.1]{gos}, we
obtain \eqref{eq:C_I-Omega}. 

%% A note of caution: when we apply \cite[Theorem~5.1]{gos} to our
%% situation, we need to observe that for any $w_1,w_2\in \cW_{I_0}$,
%% $w_i(wv_0)^{I_0}=(w_iwv_0)^{I_0}$, and the sets
%% $KA^{I_0,+}w_iA_{I_0}(wv_0)^{I_0}$ are either identical or their
%% intersection consist of singular sets of measure zero;
%% moreover the functions $\delta_I(a)$ and $\exp(2\rho(b))$ in that theorem
%% are invariant under conjugation by $w_i$.
%\end{proof}
\end{rem}
\subsubsection{Upper estimate of volume for $(J,c)$-singular elements
  in $S_T(\Omega,w)$}

For $c>0$, $I\subset \Delta_\sigma$, and a bounded
measurable $\Omega\subset KM_I$, we set
$$
V_{I,w}(c)=\left\{m a wv_0:\, m\in\Omega,\, a\in A_I^+\hbox{ with
    $\alpha(\log a)\le c$ for some
    $\alpha\in\Delta_\sigma-I$}\right\}.
$$
Note that this set is the set of $(J,c)$-singular elements for $J=\Delta_{\sigma}\setminus I$. 

\begin{Prop}\label{p:walls}
  For small $c>0$ and sufficiently large $T>0$,
$$
\vol(V_{I,w}(c)\cap B_T)\ll c\cdot T^{a_I}(\log T)^{b_I-1}.
$$
\end{Prop}

\begin{proof}
  For $\alpha\in\Delta_\sigma$, set
$$
U_c(\alpha)=\{a\in A_I^+: \alpha(\log a)\le c\}.
$$ 
There exists $\delta>1$ such that $m^{-1}B_T\subset B_{\delta T}$
for all $T>0$.  By (\ref{eq:volume}), this gives the estimate
\begin{align}\label{eq:perp}
  \vol(V_{I,w}(c)\cap B_T) &\ll \sum_{\alpha\in \Delta_\sigma-I} \int_{a\in
    A_I^+\cap U_c(\alpha): \|a v_0\|<\delta T} \xi_I(a)da.
\end{align}
Now we use the volume computation from \cite{gos} (see the proof of
Theorem 6.4 in \cite{gos}) to show for every nonnegative $f\in
C_c(W)$,
$$
\int_{A_I^+\cap U_c(\alpha)} f(av_0/T)\xi_I(a)da \ll
\left(\int_{A_I^+\cap U_c(\alpha)} f(a
  v^{I_0})\tilde{\xi}_I(a)da\right)\cdot T^{a_I}(\log T)^{b_I-1},
$$ where $I\subset I_0\subset \Delta_\sigma$, $v^{I_0}\in W$ and
$\tilde\xi_I\in C(A^+)$ are as defined in
section~\ref{subsec:volume_estimate}. By \cite[Corollary~4.7]{gos} the
projection of $v_0^{I_0}$ on the $\lambda_\iota$-eigenspace is
nonzero, and the map $A^+\to\R: a\mapsto \lambda_\iota(a)$ is
proper. Therefore the map $A_I^+\to W:a\mapsto av_0^{I_0}$ is proper.
This implies that there exists a compact $L\subset A^+_I$ such that
\[
L\supset \{a\in A^+_I:av_0^{I_0}\in\supp f\}.
\]
Then
\begin{align*}
  \int_{A_I^+\cap U_c(\alpha)} f(av_0/T)\xi_I(a)da &\ll \max (f)\cdot
  \vol(L\cap U_c(\alpha))
  \cdot T^{a_I}(\log T)^{b_I-1}\\
  &\ll_f c\cdot T^{a_I}(\log T)^{b_I-1}.
\end{align*}
Taking a function $f$ satisfying $\chi_{B_1}\le f $, we obtain
\begin{align*}
  \int_{a\in A_I^+\cap U_c(\alpha): \|av_0\|<T} \xi_I(a)da &\ll
  \left(\int_{A_I^+\cap U_c(\alpha)}
    f(a v)\tilde{\xi}_I(a)da\right)\cdot T^{a_I}(\log T)^{b_I-1}\\
  &\ll_f c\cdot T^{a_I}(\log T)^{b_I-1}.
\end{align*}
Therefore, by (\ref{eq:perp}),
\begin{align*}
  \vol(V_{I,w}(c)\cap B_T) & \ll c\cdot (\delta T)^{a_I}(\log (\delta
  T))^{b_I-1}.
\end{align*}
This completes the proof.
\end{proof}

The following corollary of Theorem~\ref{th:sw2} will be used in the
proof of Theorem~\ref{cor:counting}:

\begin{cor}\label{c:3}
  Let $\Delta_\sigma=I\sqcup J$ and $B$ be a bounded subset of $KM_I$.
  Then given $c>0$, there exist $\ell>1$ and $\e_0>0$ such that for
  every $(J,c)$-regular $g=bah\in B A_IH$ and $0<\e<\e_0$,
$$
\Cal O_\e g \subset ( K\cap \mathcal{O}_{\ell\e})b (M_I\cap
\mathcal{O}_{\ell\e})(A_I\cap \mathcal{O}_{\ell\e})aH.
$$
\end{cor}

\begin{proof}
  Let $b=km$ for $k\in K$ and $m\in M_I$. Note that $m\in KB\cap M_I$,
  which is bounded. By \eqref{eq:M_I-Cartan} there exist $k_0\in
  M_I\cap K$, $a_0\in A^I$ and $h_0\in M_I\cap H$ such that
  $m=k_0a_0h_0$.  By Theorem~\ref{th:sw2},
$$
\Cal O_\e g\subset ( K\cap \mathcal{O}_{\ell\e})kk_0 (M_I\cap
\mathcal{O}_{\ell\e})(A_I\cap \mathcal{O}_{\ell\e})a_0aH.
$$
There exists $\sigma>1$ such that for every $k\in K$ and small $\e>0$,
$k\mathcal{O}_\e k^{-1}\subset \mathcal{O}_{\sigma \e}$. Hence,
$$
\Cal O_\e g \subset ( K\cap \mathcal{O}_{\ell\e})k (M_I\cap
\mathcal{O}_{\sigma \ell\e})k_0a_0h_0(A_I\cap \mathcal{O}_{\ell\e})aH.
$$
There exists $\eta>1$ such that for every $m\in KB$ and small $\e>0$,
$m^{-1}\mathcal{O}_\e m\subset \mathcal{O}_{\eta \e}$. Hence,
$$
\Cal O_\e g \subset ( K\cap \mathcal{O}_{\ell\e})km (M_I\cap
\mathcal{O}_{\eta\sigma \ell\e})(A_I\cap \mathcal{O}_{\ell\e})aH
$$
as required.
\end{proof}

\begin{proof}[Proof of Theorem~\ref{cor:counting}] Due to
  Remark~\ref{rem:Weyl} without loss of generality, we may assume that
  $w=e$. We will denote $S_T(\Omega,e)$ by $S_T(\Omega)$. 
  
  Let $c,\e\in (0,1)$.

  Let $s:KM_I/(M_I\cap H)\to KM_I$ be a measurable section such that
  $s(\Omega)$ is bounded and measurable.  For neighborhoods $U_1$ of
  $e$ in $K$ and $U_2$ of $e$ in $M_I$ , we set
  \begin{align*}
    \Omega^+&=U_1 s(\Omega) U_2(M_I\cap H),\\
    \Omega^- &=\bigcap_{u_1\in U_1,u_2\in U_2} u_1
    s(\Omega)u_2(M_I\cap H).
  \end{align*}
  One can check that as $U_1$ and $U_2$ shrink to $\{e\}$, we have
$$
\Omega^+\,\downarrow\, \bar\Omega\quad\quad\hbox{and}\quad\quad
\Omega^-\,\uparrow\, \hbox{int}(\Omega).
$$
Since $\vol(\partial\Omega)=0$, we have $\vol(\Omega^+-\Omega^-)\to
0$.  Hence, it follows from (\ref{eq:asympt}) that we can choose $U_1$
and $U_2$ so that
\begin{equation}\label{eq:sigma}
  C_I(\Omega^+)-C_I(\Omega^-)<\e.
\end{equation}
% Moreover, taking $U$ with boundary of measure zero, we have
% $\vol(\partial\Omega_i^\pm)=0.$
Fix a set $\tilde \Omega\supset \Omega$ such that $\bar\Omega\subset
\hbox{int}(\tilde\Omega)$, set
$$
V_I=\Omega A_I^+v_0\quad\hbox{and}\quad \tilde V_I=\tilde\Omega
A_I^+v_0,
$$ 
and define $V_I(c)=V_{I,e}(c)$ and $\tilde V_I(c)=\tilde V_{I,e}(c)$ as in
Proposition~\ref{p:walls}. We can choose $U_1$ and $U_2$ so that
$\Omega^+\subset \tilde\Omega$.

We claim that there exists a neighborhood $\mathcal{O}'$ of $e$ in $G$
such that
\begin{equation}\label{eq:claim}
  \mathcal{O}'\cdot S_T(\Omega)\subset S_{(1+\e)T}(\Omega^+)\cup
  (\tilde V_I(c)\cap B_{(1+\e)T}).
\end{equation}
By Corollary~\ref{c:3}, there exists a neighborhood $\mathcal{O}_1$
such that
$$
\mathcal{O}_1^{-1}\cdot (V_I-\tilde V_I(c))\subset \tilde
V_I-V_I({c/2}).
$$
This implies that
$$
\mathcal{O}_1\cdot V_I({c/2})\subset \tilde V_I(c).
$$
Also, by Corollary~\ref{c:3} and continuity of operator norm, there
exists a neighborhood $\mathcal{O}_2$ of $e$ in $G$ such that for
every $v=m av_0\in V_I-V_I(c/2)$,
$$
\mathcal{O}_2v\subset (U_1 m U_2)A_I^+v_0
$$
and
$$
\mathcal{O}_2\cdot B_T\subset B_{(1+\e)T}.
$$
Hence,
$$
\mathcal{O}_2\cdot (S_T(\Omega)-V_I({c/2})) \subset
S_{(1+\e)T}(\Omega^+).
$$
Setting $\mathcal{O}'=\mathcal{O}_1\cap \mathcal{O}_2$, we deduce the
claim (\ref{eq:claim}).

Similar argument shows there exists a neighborhood $\mathcal{O}''$ of
$e$ in $G$ such that
\begin{equation}\label{eq:claim2}
  S_{(1-\e)T}(\Omega^-)\subset \left( \bigcap_{g\in
      \mathcal{O}''} gS_T(\Omega)\right) \cup \tilde V_I(c).
\end{equation}

Combining (\ref{eq:claim}) and (\ref{eq:claim2}), we deduce that for
$\mathcal{O}=\mathcal{O}'\cap\mathcal{O}''$,
\begin{align}\label{eq:est}
  &\vol(\mathcal{O}\cdot \partial S_T(\Omega)) \le
  \vol\left(\mathcal{O}S_T(\Omega) - \cap_{g\in \mathcal{O}}
    gS_T(\Omega)\right)\\
  &\le \vol(S_{(1+\e)T}(\Omega^+))-
  \vol(S_{(1-\e)T}(\Omega^-))+\vol(\tilde V_I(c)\cap B_{(1+\e)T}).
  \nonumber
\end{align}
% There exist measurable sets $\tilde\Omega^+,\tilde\Omega^-\subset K$
% with boundaries of measure zero such that
%$$
%\tilde\Omega^-\subset \Omega^-\subset \Omega^+\subset
% \tilde\Omega^+,\;\; \vol(\Omega^- -\tilde\Omega^-)<\e,\;\;
% \vol(\tilde\Omega^+- \Omega^+)<\e.
%$$
By Proposition~\ref{p:walls},
$$
\limsup_{T\to\infty} \frac{\vol(\tilde V_I(c)\cap
  B_{(1+\e)T})}{T^{a_I}(\log T)^{b_I-1}}\ll c.
$$
By Proposition~\ref{p:volume},
\begin{align*}
  \lim_{T\to\infty} \frac{\vol(S_{(1+\e)T}(\Omega^+))}{T^{a_I}(\log
    T)^{b_I-1}}
  &= (1+\e)^{a_I}C_I(\Omega^+),\\
  \lim_{T\to\infty} \frac{\vol(S_{(1-\e)T}(\Omega^-))}{T^{a_I}(\log
    T)^{b_I-1}} &= (1-\e)^{a_I}C_I(\Omega^-).
\end{align*}
Hence, it follows from (\ref{eq:est}) and (\ref{eq:sigma}) that
\begin{align*}
  \limsup_{T\to\infty} \frac{\vol(\mathcal{O}\cdot \partial
    S_T(\Omega))}{T^{a_I}(\log T)^{b_I-1}}
  &\ll (1+\e)^{a_I}C_I(\Omega^+)-(1-\e)^{a_I}C_I(\Omega^-)+c\\
  &\ll \e+c.
\end{align*}
Since $\e$ and $c$ can be taken arbitrary small, this proves that the
family of sets $S_T(\Omega)$ is well-rounded.  Hence, it follows from
\cite{drs,em} that
$$
\#(\Gamma v_0\cap S_T(\Omega))\sim_{T\to\infty} \vol(S_T(\Omega)).
$$
This proves the theorem.
\end{proof}

\begin{proof}[Proof of Theorem~\ref{th:quad}] 
  To deduce Theorem~\ref{th:quad} from Theorem~\ref{cor:counting}, we
  observe that (see \cite[\S2.3]{gos}):
$$
\mathcal{Q}_W\simeq\bigcup_{p+q=d}
\hbox{SL}_{d}(\mathbb{R})/\hbox{SO}(p,q),\quad d=\dim W,
$$
and $\hbox{SL}_d(\mathbb{R})/\hbox{SO}(p,q)$ is an affine symmetric
space. We set
\begin{align*}
  G&=\SL_{d}(\mathbb{R}),\\
  K&=\SO(d),\\
  A&=\{\diag(s_1,\ldots,s_d): s_i\in\mathbb{R}^+,\,s_1\cdots s_d=1\},\\
  H&=\SO(p,q).
\end{align*}
Then we have the generalized Cartan decomposition $G=KAH$. The set of
simple roots is
\begin{gather*}
\Delta_\sigma=\{\alpha_i(s)=s_is_{i+1}^{-1}: i=1,\ldots,d-1\}, \quad
\text{and} \\
A^+=\{\diag(s_1,\dots,s_d):s_1>\dots>s_d>0\}.
\end{gather*}

In view of \eqref{eq:W} and \eqref{eq:N_T}, set 
\[
i_k=\sum_{i=1}^k \dim W_i, \quad 1\leq k\leq n.
\]
Let $I=\Delta_{\sigma}\setminus
\{\alpha_{i_1},\dots,\alpha_{i_n}\}$. Then
$$
M_I\simeq \SL_{i_1}(\mathbb{R})\times
\SL_{i_2-i_1}(\mathbb{R})\times \cdots
\times\SL_{d-i_n}(\mathbb{R}),
$$
and $A_I$ is the centralizer of $M_I$ in $A$.

Since the set of integral quadratic forms in the question is a finite
union of $\SL_d(\mathbb{Z})$-orbits, we conclude that the proof of the
theorem reduces to the computation of the asymptotics of
$\#(\SL_d(\mathbb{Z}) \mbq_0\cap S_T(\Omega\Omega'))$ where
$\mbq_0\in\mathcal{Q}_W(\mathbb{Z})$.  This shows that
Theorem~\ref{th:quad} is a particular case of
Theorem~\ref{cor:counting}; it may be noted that since $d\geq 3$ the
subgroups $\SO(p,q)$ are semisimple and $\SO(p,q)\cap\SL-d(\Z)$ is a
lattice in $\SO(p,q)$.

It remains to compute the parameters $a_I$ and $b_I$, which are
determined by the volume asymptotics in Proposition~\ref{p:volume}.

% Note that the function $\xi_I$,
% defined in (\ref{eq:xi}), decomposes as a linear combination of
% functions $\exp(\chi(a))$ where $\chi$'s are characters of $A_I$. The
% maximal character in this decomposition is
If we restrict the character $2\rho$, which is the sum of all roots in
$\Sigma_\sigma^+$, then we get
$$ 
\rho|_{{\text{Lie}(A_I)}}=\sum_{k=1}^n u_{i_k} \alpha_{i_k}, \quad \text{where
  $u_{i_k}=i_k(d-i_k)$.}
$$
The highest weight, say $\lambda_\iota$, of the representation of $\hbox{SL}_d(\mathbb{R})$
on the space of quadratic forms restricted to Lie($A_I$) is
$$
\lambda_\iota|_{\text{Lie}(A_I)}=\sum_{k=1}^n m_{i_k}\alpha_{i_k}\quad \text{where $m_i=2(d-i_k)/d$.}
$$
By \eqref{eq:a_I} and \eqref{eq:b_I},
\begin{align*}
  a_I&=\max\left\{\frac{u_{i_k}}{m_{i_k}}:\, 1\leq k\leq n\right\}=di_n/2,\\
  b_I&=\#\left\{i_k:\,1\leq k\leq n,\,
    \frac{u_{i_k}}{m_{i_k}}=a_I\right\}=1.
\end{align*}
This proves the theorem.
\end{proof}

\section{Another version of the strong wavefront lemma}\label{sec:last}

%The statements of strong wavefront lemma given by Theorem~\ref{th:sw}
%and Theorem~\ref{th:sw2}
%are suitable for problems involving finite-dimensional representations
%of Lie groups. In this section we will obtain a
%version which seems more satisfactory from the geometric or group-theoretic viewpoints. 
 
In this section, we obtain a version of the strong wavefront lemma
for a generalized Cartan decomposition with a different Weyl chamber
$\tilde A^+$ defined below.

Let $G^{\sigma\theta}=\{g\in G:\sigma\theta(g)=g\}$, the symmetric
subgroup associated to the involution $\sigma\theta$ of $G$ and
$\la{g}^{\sigma\theta}$ be the associated Lie subalgebra. Then $A$ is
the maximal $\R$-split Cartan subalgebra of $G^{\sigma\theta}$.
Set 
\[
\tilde\Sigma_{\sigma,\theta}
=\{\alpha\in\Sigma_\sigma:\la{g}^\alpha\cap\la{g}^{\sigma\theta}\neq\{0\}\}
\quad\hbox{and}\quad\tilde\Sigma_{\sigma,\theta}^+=\Sigma_\sigma^+\cap
\tilde\Sigma_{\sigma,\theta}.
\]
Then $\tilde\Sigma_{\sigma,\theta}$ is
a root system on $A$, and we denote by  $\tilde \Delta_{\sigma,\theta}\subset
\tilde\Sigma_{\sigma,\theta}^+$ the set of simple roots on $A$.
Let $\tilde{A^+}$ denote the associated closed Weyl chamber of
$A$. Then $A^+\subset \tilde{A^+}$. Also the following generalized
Cartan decomposition holds:
\[
G=K\tilde{A}^+H.
\]
Note that $\tilde A^+\ne A^+$ in general (see \cite[p. 109]{hs}).

Given $c>0$, an element $g=kah\in K\tilde{A^+}H$ is called $c$-regular for
$\tilde\Delta_{\sigma,\theta}$ if $\alpha(a)>c$ for all
$\alpha\in\tilde\Delta_{\sigma,\theta}$. 

\begin{thm}[Strong wavefront Lemma-III] \label{th:sw3} Given $c>0$, there exist
  $\ell>1$ and $\e_0>0$ such that for every $g=kah\in K\tilde{A^{+}}H$
  which is $c$-regular for $\tilde{\Delta}_{\sigma,\theta}$ and
  every $0<\e<\e_0$, 
\[
\cO_\e g\subset (K\cap \cO_{\ell \e})k \cdot (A\cap \cO_{\ell \e})a\cdot
(H\cap\cO_{\ell \e})h.
\]
\end{thm}

This result is stronger than Theorem~\ref{th:sw} because any
$c$-regular element is also $c$-regular for
$\tilde\Delta_{\sigma,\theta}$, but the converse
implication does not hold in general.

Now we consider the situation involving singular elements.  Let
$\tilde I\subset\tilde \Delta_{\sigma,\theta}$ and $\tilde J=\tilde
\Delta_{\sigma,\theta}\setminus \tilde I$. For $c>0$, we say that an
element $g=kah\in KA^+H$ is $(\tilde J,c)$-regular if $\alpha(\log
a)>c$ for all $\alpha\in \tilde J$. Let $A_{\tilde I}=\exp(\ker \tilde
I)$.  Let $M^{\sigma\theta}_{\tilde I}$ denote the analytic semisimple
subgroup of $G^{\sigma\theta}$ whose Lie algebra is generated by
$\la{g}^{\pm\beta}\cap \la{g}^{\sigma\theta}$ for all $\beta\in
\Sigma^+_{\sigma,\theta}\cap \langle \tilde I\rangle$.  Then
$M^{\sigma\theta}_{\tilde I}$ is contained in the centralizer of
$A_{\tilde I}$, and
\[
G=KM^{\sigma\theta}_{\tilde I}A_{\tilde I}^+H,
\]
where $A_{\tilde I}^+=\tilde{A^+}\cap A_{\tilde I}$.

\begin{thm}[Strong wave front lemma-IV]\label{th:sw4}
  Given $c>0$, there exist $\ell>1$ and $\e_0>0$ such that for every
  $\tilde I\subset\tilde \Delta_{\sigma\theta}$, $\tilde J=\tilde
  \Delta_{\sigma,\theta}\setminus\tilde I$, $g=kah\in
  K\tilde{A^+}H$ which is $(\tilde J,c)$-regular,  and $0<\e<\e_0$,
  \[
  \Cal O_\e\cdot g \subset ( K\cap \mathcal{O}_{\ell\e})k\cdot
  (M^{\sigma\theta}_{\tilde I}\cap \mathcal{O}_{\ell\e})\cdot
  (A_{\tilde I}\cap \mathcal{O}_{\ell\e})a\cdot (H\cap
  \mathcal{O}_{\ell \e})h.
  \]
\end{thm}

This result strengthens Theorem~\ref{th:sw2}.

\begin{lem} \label{l:sum-conj} For any $a\in A$,
  \[
  \la{g}=\la{q}\oplus(\la{k}\cap\la{h}) \oplus \Ad a(\la{p}\cap\la{h}).
  \]
\end{lem}

\begin{proof} Since $\la{g}=\la{q}\oplus (\la{k}\cap\la{h})\oplus
  (\la{p}\cap\la{h})$, it is enough to to show that
  \[
  \Ad a(\la{p}\cap\la{h})\cap \left(\la{q}+(\la{k}\cap\la{h})\right)=\{0\}.
  \]
To prove this, let $X\in\la{p}\cap\la{h}$ such that $\Ad
  a(X)\in\la{q}\oplus (\la{k}\cap\la{h})$. Therefore,
  \begin{align*}
    \sigma(\Ad a(X))&=\Ad \sigma(a) (\sigma(X))=(\Ad a)\inv (X),\\
    \theta(\Ad a(X))&=\Ad \theta(a) (\theta(X))=(\Ad a)\inv (-X),
  \end{align*}
and 
\begin{equation}\label{eq:l41}
\sigma(\Ad a(X))=-\theta(\Ad a(X))=(\Ad a)\inv (X).
\end{equation}
  Now we write $\Ad a(X)=Y_1+Y_2+Y_3$, where $Y_1\in
  \la{q}\cap\la{k}$, $Y_2\in\la{q}\cap\la{p}$, and $Y_3\in
  \la{k}\cap\la{h}$. Then
  \begin{align*}
    \sigma(\Ad a(X))&=-Y_1-Y_2+Y_3,\\
    \theta(\Ad a(X))&=Y_1-Y_2+Y_3,
  \end{align*}
  and it follows from (\ref{eq:l41}) that $Y_2=0$ and $Y_3=0$. Hence, $\Ad
  a(X)\in\la{k}\cap\la{q}$ and $\sigma(\Ad a(X))=-\Ad a(X)$. Then by
  (\ref{eq:l41}),
$$
(\Ad a)^2(X)=-X.
$$
If $X\ne 0$, this gives a contradiction because $\Ad a$ is self-adjoint.
\end{proof}

As a consequence of the above lemma, we obtain the following:

\begin{cor} \label{cor:open-conj} Given $c>0$ there exist $\ell>1$
  and $\e_0>0$ such that for any $a\in A$ such that
  $\abs{\alpha(\log a)}\leq c$ for all $\alpha\in\Delta_\sigma$, and
  any $0<\e<\e_0$, we have
  \[
  \cO_\e a \subset (\cO_{\ell\e}\cap
  K)(\cO_{\ell\e}\cap\exp(\la{p}\cap\la{q}))a(\cO_{\ell\e}\cap\exp(\la{p}\cap\la{h})).
  \]
  \qed
\end{cor}

\begin{proof}[Proof of Theorem~\ref{th:sw4}]  Let $w\in \cW$ be such that
$waw\inv =b\in A^+$. We set
$$
I=\{\alpha\in\Delta_\sigma:\alpha(\log b)<c/n_0\}\quad\hbox{and}\quad
J=\Delta_\sigma\setminus I.
$$
where $n_0\in\N$ is such that any positive root is a
sum of at most $n_0$ simple roots counted with multiplicity.
We apply Theorem~\ref{th:sw2} to the involution $\sigma_w:=i_w\circ\sigma\circ i_w\inv$
in place of $\sigma$. Since the element $(kw\inv)b(whw\inv)$ is
$(J,c/n_0)$-regular,
\begin{align}\label{eq:1}
  \cO_\e (kah)=&\cO_\e(kw\inv)b(whw\inv)w \\
  \subset&(\cO_{\ell\e}\cap K)(kw\inv) (\cO_{\ell\e}\cap
  M_I)
  (\cO_{\ell\e}\cap A_I)b(\cO_{\ell\e}\cap wHw\inv)(whw\inv)w\nonumber\\
  =& (\cO_{\ell\e} \cap K)k(w\inv \cO_{\ell\e}{w} \cap {w}\inv
  M_Iw)(w\inv \cO_{\ell\e}w \cap w\inv A_I w)a\nonumber\\
&\times (w\inv
  \cO_{\ell\e}w\cap H)h.\nonumber
\end{align}
There exists $\ell_1>1$ such that 
$$
w\inv
\cO_{\ell\e}w\subset\cO_{\ell_1\e}
$$
for all
$0<\e<\e_0$.
Since $M_I$ is $\sigma_w$- and $\theta$-stable, $M_I^w:=w\inv M_I w$
is $\sigma$- and $\theta$-stable and $A=(A\cap M_I^w)(w\inv A_Iw)$.
 Let $a_1\in A\cap M_I^w$ be such that $a\in
a_1(w\inv A_I w\inv)$.  We now apply Corollary~\ref{cor:open-conj} to
$M_I^w$ in place of $G$, and conclude that for some $\ell_2\geq
\ell_1$,
\begin{align}\label{eq:2}
  (\cO_{\ell_1\e}\cap M_I^w)a_1
  \subset& (\cO_{\ell_2\e}\cap K\cap M_I^w)
  (\cO_{\ell_2\e}\cap \exp(\la{p}\cap\la{q}) \cap M_I^w) a_1\\
  &\times (\cO_{\ell_2\e}\cap\exp(\la{p}\cap\la{h})\cap M_I^w).\nonumber
\end{align}
Since $M_I^w$ commutes with $w\inv
A_I w$, combining (\ref{eq:1}) and (\ref{eq:2}), we obtain that for some $\ell_3\geq
\ell_2$
\begin{align}\label{eq:3}
 \cO_\e (kah) \subset
  (\cO_{\ell_3\e}\cap K)k (\cO_{\ell_3\e}\cap
  \exp(\la{p}\cap\la{q})\cap M_I^w) (\cO_{\ell_3\e}\cap w\inv
  A_Iw) a (\cO_{\ell_3\e}\cap H) h.
\end{align}
By the definition of $I$, each eigenvalue of $\ad (\log b)$ on the Lie
algebra of $M_I$ is at most $c$. Hence every eigenvalue of $\ad (\log
a)$ on the Lie algebra of $M_I^w$ is at most $c$. Since $a$ is given
to be $(\tilde J,c)$-regular, we conclude that
\[
M_I^w\cap \exp(\la{p}\cap\la{q})\subset M_I^w\cap
G^{\sigma\theta}\subset M_{\tilde I}.
\]
Therefore, the conclusion of the theorem follows from
(\ref{eq:3}).
\end{proof}

Note that Theorem~\ref{th:sw3} follows from Theorem~\ref{th:sw4}.

\end{document}